\documentclass[12pt]{amsart}
\usepackage[top=30truemm,bottom=30truemm,left=25truemm,right=25truemm]{geometry}
\usepackage{mathrsfs}

\usepackage{color}
\usepackage{bm}
\usepackage{amsfonts,amssymb}
\usepackage{dsfont}
\usepackage{amscd}
\usepackage{extarrows}
\usepackage{amsmath}
\usepackage{mathrsfs}
\usepackage{enumerate}
\usepackage{amscd}
\usepackage[all]{xy}
\usepackage[pagebackref,colorlinks]{hyperref}

\newtheorem{thm}{Theorem}[section]
\newtheorem{defn}[thm]{Definition}
\newtheorem{lem}[thm]{Lemma}
\newtheorem{cor}[thm]{Corollary}
\newtheorem{prop}[thm]{Proposition}
\newtheorem{rem}[thm]{Remark}

\newcommand{\be}{\begin{equation}}
	\newcommand{\ee}{\end{equation}}
\newcommand{\bea}{\begin{eqnarray}}
	\newcommand{\eea}{\end{eqnarray}}
\newcommand{\ben}{\begin{eqnarray*}}
	\newcommand{\een}{\end{eqnarray*}}
\newcommand{\bt}{\begin{split}}
	\newcommand{\et}{\end{split}}
\newcommand{\bet}{\begin{equation}}

\begin{document}
		\title[Range of Complex Monge-Amp\`ere operators]{On the Range of a class of  Complex Monge-Amp\`ere  operators on compact Hermitian manifolds}
		
		%\author[F. Deng]{Fusheng Deng}
		%\address{Fusheng Deng: \ School of Mathematical Sciences, University of Chinese Academy of Sciences\\ Beijing 100049, P. R. China}
		%\email{fshdeng@ucas.ac.cn}
		%
		%\author[J. Ning]{Jiafu Ning}
		%\address{Jiafu Ning: \ Department of Mathematics, Central South University, Changsha, Hunan 410083, P. R. China.}
		%\email{jfning@csu.edu.cn}
		\author[Y. Li]{Yinji Li}
		\address{Yinji Li:  Institute of Mathematics\\Academy of Mathematics and Systems Sciences\\Chinese Academy of
			Sciences\\Beijing\\100190\\P. R. China}
		\email{1141287853@qq.com}
		\author[Z. Wang]{Zhiwei Wang}
		\address{Zhiwei Wang: Laboratory of Mathematics and Complex Systems (Ministry of Education)\\ School of Mathematical Sciences\\ Beijing Normal University\\ Beijing 100875\\ P. R. China}
		\email{zhiwei@bnu.edu.cn}
		\author[X. Zhou]{Xiangyu Zhou}
		\address{Xiangyu Zhou: Institute of Mathematics\\Academy of Mathematics and Systems Sciences\\and Hua Loo-Keng Key
			Laboratory of Mathematics\\Chinese Academy of
			Sciences\\Beijing\\100190\\P. R. China}
		\address{School of
			Mathematical Sciences, University of Chinese Academy of Sciences,
			Beijing 100049, P. R. China}
		\email{xyzhou@math.ac.cn}
		
		\begin{abstract}
			
			Let $(X,\omega)$ be a compact Hermitian manifold of complex dimension $n$. Let $\beta$ be a smooth real closed $(1,1)$ form such that there exists a function $\rho \in \mbox{PSH}(X,\beta)\cap L^{\infty}(X)$. We study the range of the complex non-pluripolar Monge-Amp\`ere operator $\langle(\beta+dd^c\cdot)^n\rangle$ on weighted Monge-Amp\`ere energy classes on $X$. In particular, when $\rho$ is assumed to be continuous,  we give a complete characterization of the range of the complex Monge-Amp\`ere operator on the class $\mathcal E(X,\beta)$, which is the class of all $\varphi \in \mbox{PSH}(X,\beta)$ with full Monge-Amp\`ere mass, i.e. $\int_X\langle (\beta+dd^c\varphi)^n\rangle=\int_X\beta^n$.
		\end{abstract}
		
		\thanks{}

		\maketitle
	%	\tableofcontents
\section{Introduction}

Let $(X, \omega)$ be a compact Hermitian manifold of complex dimension $n$, with a Hermitian metric $\omega$. Let $\beta$ be a smooth real closed $(1,1)$ form. A function $u:X\rightarrow [-\infty,+\infty)$ is called quasi-plurisubharmonic if locally $u$ can be written as the sum of a smooth function and a plurisubharmonic function. A $\beta$-plurisubharmonic ($\beta$-psh for short) function $u$ is defined as a quasi-plurisubharmonic function satisfying $\beta+dd^cu\geq 0$ in the sense of currents. The set of all $\beta$-psh functions on $X$ is denoted by $\mbox{PSH}(X,\beta)$.

Suppose that there exists a function $\rho\in \mbox{PSH}(X,\beta)\cap L^{\infty}(X)$. We define the $(\beta+dd^c\rho)$-psh function to be a function $v$ such that $\rho+v\in \mbox{PSH}(X,\beta)$, and denote by $\mbox{PSH}(X,\beta+dd^c\rho)$ the set of all $(\beta+dd^c\rho)$-psh functions.  In \cite{LWZ23}, the authors solve the following  degenerate complex Monge-Amp\`ere equation 
	\[(\beta+dd^c\varphi)^n=f\omega^n,\]
with $\varphi\in \mbox{PSH}(X,\beta)\cap L^\infty(X)$ and  $0\leq f\in L^p(X,\omega^n)$, $p>1$, such that  $\int_Xf\omega^n=\int_X\beta^n$.  

Building on this result, we further investigate the complex non-pluripolar Monge-Ampère equation:$$\langle (\beta+dd^c\varphi)^n\rangle=\mu, \quad \varphi \in \mbox{PSH}(X,\beta),$$ where $\mu$ is a positive non-pluripolar Radon measure on $X$, and $\langle (\beta+dd^c\varphi)^n\rangle$ is the non-pluripolar product (see \cite{BEGZ10}).

When the manifold $X$ is assumed to be compact K\"ahler and $\beta$ is assumed to be a K\"ahler metric or in a semi-positive big class, the complex Monge-Amp\`ere equations mentioned above have been extensively studied in recent years. These equations have important applications in K\"ahler geometry (e.g., see \cite{GZ07,EGZ09,BEGZ10,EGZ11,BBGZ13,EGZ17} and references therein) as well as in complex dynamics (see \cite{Sib99}).

To state our results, we introduce some notions. Let $A$ and $a$ be positive constants. For a positive Radon measure $\mu$ on $X$, we say $\mu\in \mathcal{H}(A,a,\beta)$ if for every Borel set $E$,
$$\mu(E)\leq A \text{Cap}_{\beta}(E)^a,$$
where
\[\text{Cap}_{\beta}(E):=\sup\left\{\int_E(\beta+dd^cv)^n:v\in \text{PSH}(X,\beta),\rho \leq v \leq \rho+1 \right\}.\]
By definition, if $\mu\in \mathcal H(A,a,\beta)$, then $\mu$ is non-pluripolar.

We denote by $\mathcal E(X,\beta)$ the class of all $\varphi \in \text{PSH}(X,\beta)$ with full Monge-Amp\`ere mass, i.e.,
$\int_X\langle (\beta+dd^c\varphi)^n\rangle=\int_X\beta^n$. We also denote by $\mathcal E^p(X,\beta)$ the class of all $\varphi\in \mathcal E(X,\beta)$ such that $\varphi\in L^p(\langle (\beta+dd^c\varphi)^n\rangle)$.

First, we study the range of the complex non-pluripolar Monge-Amp\`ere operator $\langle(\beta+dd^c\cdot)^n\rangle$ on the class $\mathcal E^p(X,\beta)$.

\begin{thm}\label{thm: main 1-1}
	Let $(X,\omega)$ be a compact Hermitian manifold of complex dimension $n$. Let $\{\beta\}\in H^{1,1}(X,\mathbb R)$ be a real $(1,1)$-class with a smooth representative $\beta$. Assume that there is a bounded function $\rho\in \text{PSH}(X,\beta)$. Let $p\geq1$, $a>\frac{p}{p+1}$, and let $\mu\in \mathcal{H}(A,a,\beta)$ satisfy $\mu(X)=\int_X\beta^n=1$. Then there exists a unique $\varphi \in \mathcal{E}^p(X,\beta)$ such that
	$$\langle(\beta+dd^c\varphi)^n\rangle=\mu, \ \sup_X\varphi=0.$$
\end{thm}

\begin{rem}
	The above theorem is a generalization of the corresponding result in \cite[Theorem 4.2, Proposition 5.3]{GZ07}, where $X$ is assumed to be K\"ahler and $\beta$ is assumed to be a K\"ahler metric on $X$.
\end{rem}

As a special case of Theorem \ref{thm: main 1-1}, let's consider the situation when $a>1$. In this case, we can obtain the $L^\infty$-estimate of the solution $\varphi$ using the method in \cite[\S 2]{EGZ09} (see also \cite[\S 3.4]{LWZ23}).

\begin{cor}\label{cor: 1}
	Let $(X,\omega)$ be a compact Hermitian manifold of complex dimension $n$. Let $\{\beta\}\in H^{1,1}(X,\mathbb R)$ be a real $(1,1)$-class with a smooth representative $\beta$. Assume that there exists a bounded function $\rho\in \mbox{PSH}(X,\beta)$. Let $a>1$ and let $\mu\in \mathcal{H}(A,a,\beta)$ such that $\mu(X)=\int_X\beta^n=1$. Then there exists a unique $\varphi\in \mbox{PSH}(X,\beta)\cap L^{\infty}(X)$ such that
		$$(\beta+dd^c\varphi)^n=\mu, \ \sup_X\varphi=0.$$
\end{cor}
	\begin{rem}
		The above corollary is a generalization of \cite[Theorem 2.1]{EGZ09}, where $X$ is assumed to be K\"ahler and $\beta$ is assumed to be semi-positive.
	\end{rem}
	
	Furthermore, when the potential $\rho\in \mbox{PSH}(X,\beta)$ is assumed to be continuous, we give a complete characterization of the range of the complex non-pluripolar Monge-Amp\`ere operator $\langle(\beta+dd^c\cdot)^n\rangle$ on the class $\mathcal E(X,\beta)$.
	
	\begin{thm}\label{thm: main 1-2} Let $(X,\omega)$ be a compact Hermitian manifold of complex dimension $n$. Let $\{\beta\}\in H^{1,1}(X,\mathbb R)$ be a real $(1,1)$-class with a smooth representative $\beta$. Assume that there exists a continuous function $\rho$ such that $\beta+dd^c\rho\geq 0$ in the sense of currents. Then the following two conditions are equivalent:
		
		\begin{itemize}
			\item [(1)] For every non-pluripolar Randon measure $\mu$ satisfying $\mu(X)=\int_X \beta^n=1$, there exists $\varphi \in \mathcal{E}(X,\beta)$ such that
			$$\langle(\beta+dd^c\varphi)^n\rangle=\mu.$$
			\item [(2)]$Cap_{\beta}(\cdot)$ can characterize pluripolar sets in the following sense:  for any Borel set $E$, $Cap_{\beta}(E)=0$ if and only if  $E$  is pluripolar. 
		\end{itemize}
		\end{thm}
		
		\begin{rem}
			When $X$ is a compact K\"ahler manifold and $\beta$ is a K\"ahler metric, the above theorem is proved by Guedj-Zeriahi \cite[Theorem A]{GZ07}.
		\end{rem}
	The structure of this paper is organized as follows. In Section \ref{sect: pre}, we present basic properties of the non-pluripolar product and several inequalities concerning the Monge-Ampère energy. In Section \ref{sect: range}, we investigate the range of the complex non-pluripolar Monge-Ampère operator on the class $\mathcal E^p(X,\beta)$ and present the proof of Theorem \ref{thm: main 1-1}. In Section \ref{sect: char}, we provide a complete characterization of the range of the complex non-pluripolar Monge-Ampère operator on the class $\mathcal E(X,\beta)$ and complete the proof of Theorem \ref{thm: main 1-2}.

\subsection*{Acknowledgements} 
This research is supported by National Key R\&D Program of China (No. 2021YFA1002600 and No. 2021YFA1003100).  Wang and Zhou  are partially supported respectively by NSFC grants (12071035, 12288201).
The second author  was   partially supported by Beijing Natural Science Foundation (1202012, Z190003). 

\section{Preliminaries }\label{sect: pre}
Let $(X,\omega) $ be a compact Hermitian manifold of complex dimension $n$. Let $\{\beta\}\in H^{1,1}(X,\mathbb R)$ be a real $(1,1)$-class with a smooth representative $\beta$. Suppose that there is a bounded $\beta$-psh function $\rho$ such that $\beta+dd^c\rho\geq 0$ in the weak sense of currents.
%Let's first introduce the space of $\beta$-psh functions with full Monge-Amp\`ere mass:

\begin{defn}[{\cite{BEGZ10}}]\label{defn: nn product}
	
	Let  $\varphi \in \mbox{PSH}(X,\beta)$. The non-pluripolar product $ \left\langle(\beta+dd^c \varphi)^n \right\rangle$  is defined as 
	$$\left\langle(\beta+dd^c \varphi)^n \right\rangle:=\lim_{k\rightarrow \infty} \mathds{1}_{\{\varphi \textgreater \rho-k\}}(\beta+dd^c\varphi^{(k)})^n,$$
	where $\varphi^{(k)}:=\max\{\varphi,\rho-k\}$. 
	\end{defn}

\begin{rem}
	It is easy to see that $O_k:=\{\varphi>\rho-k\}$ is a plurifine open subset. For any compact subset $K$ of $X$, we have
	\[\sup_k\int_{K\cap O_k}(\beta+dd^c\varphi^{(k)})^n \leq \sup_k\int_X(\beta+dd^c\varphi^{(k)})^n=\int_X\beta^n<+\infty.\]
	Thus, the Definition \ref{defn: nn product} of non-pluripolar product is well-defined due to \cite[Definition 1.1]{BEGZ10}, and it has the basic properties in \cite[Proposition 1.4]{BEGZ10}.
\end{rem}
\begin{defn}\label{defn: full ma mass}
	Let $\varphi \in \mbox{PSH}(X,\beta)$. We say that $\varphi$ has full Monge-Amp\`ere mass if
	\[\int_X \langle(\beta+dd^c\varphi)^n\rangle=\int_X \beta^n.\]
	We denote by $\mathcal{E}(X,\beta)$ the class of $\beta$-psh functions with full Monge-Amp\`ere mass.
\end{defn}		
	
By a slight modification of the proof of \cite[Lemma 1.2]{GZ07}, we have the following lemma:

\begin{lem}\label{lem: npp basic prop}
	Let $\varphi\in \mbox{PSH}(X,\beta)$. Let $(s_k)$ be any sequence of real numbers converging to $+\infty$, such that $s_k\leq k$ for all $k\in \mathbb N$. The following conditions are equivalent:
	\begin{itemize}
		\item[(a)] $\varphi\in \mathcal E(X,\beta)$;
		\item[(b)] $(\beta+dd^c\varphi^{(k)})^n(\varphi\leq\rho-k )\rightarrow 0$;
		\item[(c)] $(\beta+dd^c\varphi^{(k)})^n(\varphi\leq\rho-s_k )\rightarrow 0$.
	\end{itemize}
\end{lem}			

From the above lemma, one can see that for any $\varphi\in \mathcal E(X,\beta)$ and bounded Borel function $b$,
\[\langle(\beta+dd^c\varphi^{(k)})^n,b\rangle \rightarrow \langle \langle(\beta+dd^c\varphi)^n\rangle,b\rangle.\]
In particular, the non-pluripolar product $\langle(\beta+dd^c\varphi)^n\rangle$ puts no mass on pluripolar sets, and 
\[\mathds{1}_{\{\varphi>\rho-k\}}\langle(\beta+dd^c\varphi)^n\rangle(B)=\mathds{1}_{\{\varphi>\rho-k\}}\langle(\beta+dd^c\varphi^{(k)})^n\rangle(B)\]
for all Borel subsets $B\subset X$.

As in \cite{BBGZ13}, a weight function is defined as a smooth increasing function $\chi:\ \mathbb{R}\rightarrow \mathbb{R}$ such that $\chi(-\infty)=-\infty$ and $\chi(t)=t$ for $t\geq0$.

Denote $\mathcal{W}^-$ as the class of all convex weight functions. It is easy to see that if $\chi\in \mathcal{W}^-$, then $\chi'\leq 1$.

\begin{rem} 
	In  \cite{GZ07}, the weight function $\chi$ is defined as a non-positive, continuously increasing function defined on $\mathbb{R}^-$, such that $\chi(-\infty)=-\infty$. In application, if $t\leq 0$, there is no need to define a weight function $\chi$ on the full real line $\mathbb{R}$, and in this case, we view a non-positive, smooth increasing function $\chi$ on $\mathbb{R}^-$, such that $\chi(-\infty)=-\infty$ as a weight function. Denote by $\mathcal{W}_h^-$ the class of all smooth, convex, increasing weight functions $\chi:\mathbb{R}^-\rightarrow \mathbb{R}^-$.
		\end{rem}

%Our aim is to study the class of $\beta$-psh functions with finite $\chi-$enegy, and the range of operator $\left\langle(\beta+dd^c \ \cdot \ )^n\right\rangle$ acting on it.
\begin{defn}
	Let $\chi$ be a weight function. For any $\varphi\in \mathcal{E}(X,\beta)$, the $\chi$-Monge-Amp\`ere energy of $\varphi$ with respect to the weight $\chi$ is defined as 
	\begin{align*}
		E_\chi(\varphi):=\int_X-\chi(\varphi-\rho)\langle(\beta+dd^c\varphi)^n\rangle.
	\end{align*}
	We denote $\mathcal{E}_\chi(X,\beta)$ as the class of all $\beta$-plurisubharmonic (PSH) functions with finite $\chi$-energy and full Monge-Amp\`ere mass.
	
	If $\chi(t)=-(-t)^p$, when $t\leq 0$, for some $p>0$, the $\chi$-energy is denoted by $E_p(\varphi)$ and the corresponding $\chi$-energy class is denoted by $\mathcal{E}^p(X,\beta)$, which is the class of all $\varphi\in \mathcal{E}(X,\beta)$ such that $\varphi\in L^p(\langle(\beta+dd^c\varphi)^n\rangle)$.
\end{defn}	   
%	\begin{rem}\label{rem: convex weight ineq}
%		For  $\chi \in \mathcal{W}^-$, it always holds that $-\chi(\varphi-C-\rho)\leq -\chi(\varphi-\rho)+C$, since convex weight function has slope $\chi'\leq 1$.  
%	\end{rem}	
		
		%One can easily see that bounded $\beta$-psh function has full Monge-Amp\`ere mass, therefore the non-pluripolar product is a generalization of Bedford-Taylor product. In the following sequel, 
		%Throughout the paper, for simplification of notations, we   write $(\beta+dd^c\varphi)^n$ instead of $\left\langle(\beta+dd^c\varphi)^n\right\rangle$. Note that if $\varphi$ is bounded,  $(\beta+dd^c\varphi)^n$ and  $\left\langle(\beta+dd^c\varphi)^n\right\rangle$ coincide.

		\begin{prop}[cf. {\cite[Proposition 2.11]{BEGZ10}}]\label{prop: chi energy}
			Let $\varphi\in \mathrm{PSH}(X,\beta)$ and $\chi$ be a (not necessarily convex) weight function. Then
			$$\varphi\in \mathcal{E}_\chi(X,\beta) \iff \sup_{k\in \mathbb{N}^*}\int_X-\chi(\varphi^{(k)}-\rho)(\beta+dd^c\varphi^{(k)})^n < +\infty.$$
				\end{prop}
		\begin{proof}
			The proof is essentially the same as the proof of  \cite[Proposition 2.11]{BEGZ10}. Without loss of generality, we may assume that $\varphi-\rho$ and $\varphi^{(k)}-\rho\leq 0$.  
			
			Assume first that $\sup_{k\in \mathbb{N}^*}\int_X-\chi(\phi^{(k)}-\rho)(\beta+dd^c\phi^{(k)})^n < +\infty$.  By the dominated convergence theorem, it suffices to prove that for $j\in\mathbb{N}^*$
			\[\int_X-\chi(\varphi^{(j)}-\rho)\langle(\beta+dd^c\varphi )^n \rangle\leq \liminf_k\int_X-\chi(\phi^{(k)}-\rho)(\beta+dd^c\phi^{(k)})^n .\]
			By the definition of non-pluripolar product, we have:
			\begin{align*}
				\int_X -\chi(\varphi^{(j)}-\rho)\langle(\beta+dd^c\varphi )^n \rangle&=\lim_k\int_X -\chi(\varphi^{(j)}-\rho)\mathds{1}_{\{\varphi>\rho-k\}}(\beta+dd^c\varphi^{(k)} )^n\\&
				\leq \liminf_k\int_X -\chi(\varphi^{(k)}-\rho)\mathds{1}_{\{\varphi>\rho-k\}}(\beta+dd^c\varphi^{(k)} )^n\\
				&\leq \liminf_k\int_X -\chi(\varphi^{(k)}-\rho)(\beta+dd^c\varphi^{(k)} )^n<+\infty.
			\end{align*}
			%By Lemma \ref{lem: npp basic prop}, $\varphi\in \mathcal E(X,\beta)$. By noting that the Borel measures $ -\chi(\phi^{(k)}-\rho)(\beta+dd^c\phi^{(k)})^n $ are uniformly bounded in mass, then they form a weakly compact sequence, with a cluster point $\nu$. Since the functions $-\chi(\phi^{(k)}-\rho)$ increase towards $-\chi(\phi -\rho)$ and $(\beta+dd^c\phi^{(k)})^n $ converges towards $(\beta+dd^c\phi )^n $, from the semi-continuity, we have that 
			%$-\chi(\varphi-\rho)\left\langle(\beta+dd^c\varphi)^n\right\rangle\leq \nu$ and $\int_X-\chi(\varphi-\rho)\left\langle(\beta+dd^c\varphi)^n\right\rangle\leq \nu(X)<+\infty$. Thus $\phi\in \mathcal{E}_{\chi}(X,\beta)$.
			
			Conversely, we assume that  $\varphi\in \mathcal E_\chi(X,\beta)$. Since 
			\[\int_{\{\varphi\leq \rho-k\}}(\beta+dd^c\varphi^{(k)})^n=\int_X(\beta+dd^c\varphi^{(k)})^n-\int_{\{\varphi> \rho-k\}}(\beta+dd^c\varphi^{(k)})^n=\int_{\{\varphi\leq \rho-k\}}\left\langle(\beta+dd^c\varphi)^n\right\rangle,\]
			then 
			\begin{align*}
				\int_X -\chi(\varphi^{(k)}-\rho)(\beta+dd^c\varphi^{(k)})^n &=-\chi(-k)\int_{\{\varphi\leq \rho-k\}}(\beta+dd^c\varphi^{(k)})^n+\int_{\{\varphi> \rho-k\}}-\chi(\varphi-\rho)(\beta+dd^c\varphi^{(k)})^n\\
				&=\int_{\{\varphi\leq \rho-k\}}-\chi(-k)\left\langle(\beta+dd^c\varphi)^n\right\rangle+\int_{\{\varphi> \rho-k\}}-\chi(\varphi-\rho)\left\langle(\beta+dd^c\varphi)^n\right\rangle\\
				&\leq \int_X-\chi(\varphi-\rho)\left\langle(\beta+dd^c\varphi)^n\right\rangle<+\infty.
			\end{align*}
			The proof of Proposition \ref{prop: chi energy} is complete.
		\end{proof}
		%Next result works well for all weight functions, it gives a characterization of $\beta-$function with finite $\chi$-energy.
		We have the following estimate on the $\chi$-energy.
		\begin{prop}[cf. {\cite[Proposition 2.8]{BEGZ10}}]\label{prop: mix chi energy ineq} 
			
			Let $\chi$ be a (not necessarily convex) weight function. For any $\varphi\in \mathcal E_\chi(X,\beta)$ and $0\leq l\leq n$, we have 
			\[
			\int_X -\chi(\varphi-\rho)\langle(\beta+dd^c\rho)^l\wedge (\beta+dd^c\varphi)^{n-l}\rangle\leq \int_X-\chi(\varphi-\rho)\langle(\beta+dd^c \varphi)^n\rangle.
			\]
		\end{prop}
		\begin{proof}
The proof is essentially the same as that of \cite[Proposition 2.8 (i)]{BEGZ10}, with the difference that we do not require the weight function $\chi$ to be convex. 
Fix $k\in \mathbb{N}^*$. By direct computation and integration by parts, we have
	\begin{align*}
				&\int_X-\chi(\varphi^{(k)}-\rho)(\beta+dd^c\rho)^l\wedge (\beta+dd^c\varphi^{(k)})^{n-l}\\
				&=\int_X-\chi(\varphi^{(k)}-\rho)(\beta+dd^c\rho)^{l-1}\wedge (\beta+dd^c\varphi^{(k)})^{n-l+1}\\
				&+\int_X\chi(\varphi^{(k)}-\rho)dd^c(\varphi^{(k)}-\rho)\wedge (\beta+dd^c\rho)^{l-1}\wedge(\beta+dd^c\varphi^{(k)})^{n-l-1} \\
				&=\int_X-\chi(\varphi^{(k)}-\rho)(\beta+dd^c\rho)^{l-1}\wedge (\beta+dd^c\varphi)^{n-l+1}\\
				&-\int_X \chi'(\varphi^{(k)}-\rho)d(\varphi^{(k)}-\rho)\wedge d^c(\varphi^{(k)}-\rho)\wedge (\beta+dd^c\rho)^{l-1}\wedge(\beta+dd^c\varphi^{(k)})^{n-l-1}\\
				&\leq \int_X -\chi(\varphi^{(k)}-\rho) (\beta+dd^c\rho)^{l-1}\wedge (\beta+dd^c\varphi^{(k)})^{n-l+1}.
			\end{align*}
	The last inequality is due to the fact that $\chi$ is an increasing function.
	
	By induction on $l$ and the proof of Proposition \ref{prop: chi energy}, as we let $k\rightarrow \infty$, we obtain:
			$$\int_X -\chi(\varphi-\rho)(\beta+dd^c\rho)^l\wedge\langle(\beta+dd^c\varphi)^{n-l}\rangle\leq \int_X-\chi(\varphi-\rho)\langle(\beta+dd^c \varphi)^n \rangle.$$
	Next, we will justify the integration by parts used above. Let us define $T=(\beta+dd^c\rho)^{l-1}\wedge (\beta+dd^c\varphi^{(k)})^{n-l+1}$. There is no need to prove the global identity of currents:
			\begin{center}
				$d[\chi(\varphi^{(k)}-\rho)d^c(\varphi^{(k)}-\rho)\wedge T]$ 
			\end{center}
			\begin{center}
				$=\chi'(\varphi^{(k)}-\rho)d(\varphi^{(k)}-\rho) \wedge d^c(\varphi^{(k)}-\rho)\wedge T+\chi(\varphi^{(k)}-\rho)dd^c(\varphi^{(k)}-\rho)\wedge T.$
			\end{center}
Let $\varphi_j, \rho_j \in \mbox{PSH}(X,\beta+\frac{1}{j}\omega)\cap C^{\infty}(X)$ be decreasing sequences towards $\varphi^{(k)},\rho$, respectively (the existence of such approximations is by Demailly's celebrated  regularization theorem \cite{Dem12}).

By the dominated convergence theorem, we can simplify the equation further:
%By dominated convergence theorem we get that
			$$\int_X\chi(\varphi^{(k)}-\rho)dd^c(\varphi^{(k)}-\rho)\wedge T=\lim_{j\rightarrow \infty}\int_X\chi(\varphi_j-\rho_j)dd^c(\varphi^{(k)}-\rho)\wedge T.$$
By using local convolution to $\varphi^{(k)},\rho$, we can prove that:
		\begin{align*}
				&d[\chi(\varphi_j-\rho_j)d^c(\varphi^{(k)}-\rho)\wedge T]\\
	&=\chi'(\varphi_j-\rho_j)d(\varphi_j-\rho_j) \wedge d^c(\varphi^{(k)}-\rho)\wedge T+\chi(\varphi_j-\rho_j)dd^c(\varphi^{(k)}-\rho)\wedge T.
			\end{align*}
Then ,	by the Stokes theorem:
		\begin{align*}&\int_X\chi(\varphi_j-\rho_j)dd^c(\varphi^{(k)}-\rho)\wedge T\\
		&=\int_X -\chi'(\varphi_j-\rho_j)d(\varphi_j-\rho_j)\wedge d^c(\varphi^{(k)}-\rho)\wedge T.
			\end{align*}
Finally, by the dominated convergence theorem:
\begin{align*}&\lim_{j\rightarrow \infty}\int_X-\chi'(\varphi_j-\rho_j)d(\varphi^{(k)}-\rho)\wedge d^c(\varphi^{(k)}-\rho)\wedge T\\
		&=\int_X-\chi'(\varphi^{(k)}-\rho)d(\varphi^{(k)}-\rho)\wedge d^c(\varphi^{(k)}-\rho)\wedge T.
			\end{align*}
To complete the proof, we need to prove the following equations:
	\begin{align*}\lim_{j\rightarrow \infty}\int_X-\chi'(\varphi_j-\rho_j)d(\varphi_j-\varphi^{(k)})\wedge d^c(\varphi^{(k)}-\rho)\wedge T=0,\end{align*}
			and
\begin{align*}\lim_{j\rightarrow \infty}\int_X-\chi'(\varphi_j-\rho_j)d(\rho_j-\rho)\wedge d^c(\varphi^{(k)}-\rho)\wedge T=0,\end{align*}
which is a consequence of the Cauchy-Schwarz inequality and the Bedford-Taylor convergence theorem \cite[Theorem (2.7)]{BT82}.
In fact, by applying the Cauchy-Schwarz inequality, we have:
\begin{align*}&\left|\int_X-\chi'(\varphi_j-\rho_j)d(\varphi_j-\varphi^{(k)})\wedge d^c(\varphi^{(k)}-\rho)\wedge T\right|\\
		&\leq \left(\int_X[\chi'(\varphi_j-\rho_j)]^2d(\varphi^{(k)}-\rho)\wedge d^c(\varphi^{(k)}-\rho)\wedge T\right)^{\frac{1}{2}}\\
		&\cdot\left (\int_Xd(\varphi_j-\varphi^{(k)}) \wedge d^c(\varphi_j-\varphi^{(k)}) \wedge T\right)^{\frac{1}{2}}.
			\end{align*}
Since $\varphi_j-\rho_j\rightarrow \varphi^{(k)}-\rho$, the first term is uniformly bounded with respect to  $j\in \mathbb{N}^*$, while the second term goes to $0$ as $j\rightarrow \infty$ by the  Bedford-Taylor convergence theorem {\cite[Theorem (2.7)]{BT82}}.  Similarly, we can prove $\lim\limits_{j\rightarrow \infty}\int_X-\chi'(\varphi_j-\rho_j)d(\rho_j-\rho)\wedge d^c(\varphi^{(k)}-\rho)\wedge T=0$. This completes the proof of Proposition \ref{prop: mix chi energy ineq}.

		\end{proof}
		
Similar to \cite{GZ07}, we can see that Proposition \ref{prop: mix chi energy ineq} implies the following proposition:

		\begin{prop}[cf. {\cite[Proposition 2.2]{GZ07}, \cite[Proposition 2.11]{BEGZ10}}]\label{prop: beta energy class}
			$$\mathcal{E}(X,\beta)=\bigcup_{\chi \in \mathcal{W}^-}\mathcal{E}_{\chi}(X,\beta).$$
		\end{prop}
		%\begin{proof}  
		%	The proof is  the same as the proof of \cite[Proposition 2.2]{GZ07}, \cite[Proposition 2.11]{BEGZ10}. The main point is that, for $\varphi\in \mathcal E(X,\beta)$, the non-pluripolar product $\left\langle(\beta+dd^c\varphi)^n\right\rangle$ is well-defined and  does not charge the pluripolar set $(\varphi=-\infty)$. Then it is easy to construct a weight function $\chi\in \mathcal W^-$, such that $\varphi\in \mathcal E_\chi(X,\beta)$.
		%\end{proof}
		The  energy class has the following basic properties:
		\begin{prop} [cf. {\cite[Proposition 2.14]{BEGZ10}}]
			The  energy classes satisfy the following properties:
			\begin{itemize}
				\item [(i)] Let $\varphi,\psi\in \mbox{PSH}(X,\beta)$ and $\chi$ be a convex weight function. If $\varphi\in\mathcal E_\chi(X,\beta)$, and $\psi\geq\phi$, then $\psi\in\mathcal E_\chi(X,\beta) $.
				\item [(ii)]The set $\mathcal{E}(X,\beta)$ is an extremal face of $\mbox{PSH}(X,\beta)$, i.e. for all  $\varphi,\psi\in\mbox{PSH}(X,\beta)$, $0< t<1$, $t\varphi+(1-t)\psi\in \mathcal{E}(X,\beta)$ implies that $\varphi,\psi \in \mathcal{E}(X,\beta)$.
				\end{itemize}
		\end{prop}
			The following comparison principle is an important tool for later use. 
				\begin{thm}[{\cite[Proposition 3.9]{LWZ23}}]\label{thm: comparison prin}
			Let $\varphi,\psi\in \mathcal{E}(X,\beta),$ then
			$$\int_{\{\varphi\textless \psi\}}\langle(\beta+dd^c\varphi)^n\rangle\leq \int_{\{\varphi\textless \psi\}}\langle(\beta+dd^c\psi)^n\rangle.$$
		\end{thm}
		Next, we introduce a monotone convergence theorem, which is used to study  the complex non-pluripolar Monge-Amp\`ere equation.
		\begin{thm}[cf. {\cite[Theorem 2.17]{BEGZ10}}] \label{thm: mono converge}
			Let $\varphi_j \in \mathcal{E}(X,\beta)$ be a sequence decreasing or increasing a.e. (in Lebesgue measure) to $\varphi \in \mathcal{E}(X,\beta)$, then
			$$\lim_{j\rightarrow \infty}\langle(\beta+dd^c\varphi_j)^n\rangle=\langle(\beta+dd^c\varphi)^n\rangle.$$
		\end{thm}
		\begin{proof}
	The proof is the same as \cite[Theorem 2.17]{BEGZ10}.
		\end{proof}
	The following proposition states that  the non-pluripolar product of the  supremum of a sequence of quasi-plurisubharmonic functions is preserved, which is important for our analysis.

		\begin{prop} [cf.{\cite[Proposition 2.20]{BEGZ10}}] \label{prop: preserve sup}
			Let $\varphi_j$ be a sequence of arbitrary $\beta-$psh functions uniformly bounded from above, and let $\varphi=(\sup_j \varphi_j)^*$. Suppose that $\mu$ is a positive measure such that $\left\langle(\beta+dd^c\varphi_j)^n\right\rangle \geq \mu$ for all $j$. If $\varphi$ has full Monge-Amp\`ere mass, then $\left\langle(\beta+dd^c\varphi)^n\right\rangle \geq \mu$.	
		\end{prop}
		\begin{proof}
		The proof is the same as \cite[Proposition 2.20]{BEGZ10}.
	\end{proof}
    We present a lemma that outlines a basic property of convex weight functions.
		\begin{lem}\label{lem: modified psh}
			Let $\chi \in \mathcal{W}^-$ and let  $\varphi \in \mbox{PSH}(X,\beta)$. Then  $\chi(\varphi-\rho)+\rho \in \mbox{PSH}(X,\beta)$.
		\end{lem}
		\begin{proof}
			
			The problem can be localized, so we only need to prove that $\chi(u-v)+v$ is plurisubharmonic in the case where $u,v \in \mbox{PSH}(B)$ and $v$ is locally bounded, where $B$ is the unit ball in $\mathbb{C}^n$.
			
			Let $u_j, v_j$ be a decreasing sequence of psh functions converging to $u,v$ respectively, obtained through standard convolutions. Since $\chi'\leq 1$, we can prove that $\chi(u_j-v_j)+v_j$ is psh by direct computation:
			\begin{align*}dd^c[\chi(u_j-v_j)+v_j]&=\chi'(u_j-v_j)dd^c(u_j-v_j)+\chi''(u_j-v_j)d(u_j-v_j)\wedge d^c(u_j-v_j)+dd^cv_j\\
				&\geq\chi'(u_j-v_j)dd^c(u_j-v_j)+dd^cv_j\geq \chi'(u_j-v_j)dd^cu_j\geq0.
			\end{align*}
			Moreover, it is easy to see that $\{\chi(u_j-v_j)+v_j\}$ is decreasing:
			\[\chi(u_{j+1}-v_{j+1})+v_{j+1}\leq \chi(u_j-v_{j+1})+v_{j+1}\leq \chi(u_j-v_j)+v_j.\]
			This completes our proof.
		\end{proof}
		
		For convex weight function $\chi$, we have the following fundamental inequality
		\begin{prop}[cf. {\cite[Lemma 2.3]{GZ07}}]\label{prop: fund inequ}
			Let  $\varphi,\psi\in\mbox{PSH}(X,\beta)\cap L^{\infty}(X)$, $\varphi\leq \psi \leq \rho$, and let $\chi$ be a convex weight function. Then   $$\int_X -\chi(\psi-\rho)(\beta+dd^c\psi)^n\leq 2^n\int_X -\chi(\varphi-\rho)(\beta+dd^c \varphi)^n. $$
		\end{prop}
		\begin{proof}
			
			The proof is a slight modification of the proof in \cite[Lemma 2.3]{GZ07}. The key point is to prove that for any positive closed current $T$ of bi-dimension $(1,1)$ on $X$,
				$$0 \leq \int_X -\chi(\varphi-\rho)(\beta+dd^c\psi)\wedge T \leq 2\int_X -\chi(\varphi-\rho)(\beta+dd^c\varphi)\wedge T.$$
				From the above inequality, we can complete the proof by observing that $\int_X -\chi(\psi-\rho)(\beta+dd^c\psi)^n \leq \int_X -\chi(\varphi-\rho)(\beta+dd^c\psi)^n$, and letting $T=(\beta+dd^c\varphi)^j\wedge(\beta+dd^c\psi)^{n-1-j}$, where $0\leq j\leq n-1$.
			By Lemma \ref{lem: modified psh}, $\chi(\varphi-\rho)+\rho$ is $\beta$-psh. By a local regularization argument and Bedford-Taylor's convergence theorem \cite{BT82}, we have that
			\begin{align*}
			&(\chi(\varphi-\rho)+\rho)dd^c\psi\wedge T-\psi dd^c(\chi(\varphi-\rho)+\rho)\wedge T
			 \\&=d[((\chi(\varphi-\rho)+\rho)d^c\psi-\psi d^c(\chi(\varphi-\rho)+\rho))\wedge T],
			\\
			&	\rho dd^c\psi\wedge T-\psi dd^c\rho\wedge T=d[(\rho d^c\psi-\psi d^c\rho)\wedge T].
			\end{align*}
				Then, by Stokes theorem, we get the following formula of integration by parts:
				$$\int_X\chi(\varphi-\rho)dd^c\psi\wedge T=\int_X \psi dd^c\chi(\varphi-\rho)\wedge T.$$
				The rest of the proof is the same as in \cite[Lemma 2.3]{GZ07}, and we omit the details here.
		
		\end{proof}
		
		\begin{rem}\label{rem: chi ene ineq}
			By Proposition \ref{prop: chi energy}, Theorem  \ref{thm: mono converge},  the above  energy inequality also holds for functions $\varphi,\psi\in\mathcal{E}_{\chi}(X,\beta)$, such that $\varphi\leq \psi\leq \rho$.
		\end{rem}
		%With above energy inequality, we can get following
		\begin{prop}[cf. {\cite[Corollary 2.4]{GZ07}}]\label{prop:char echi}
			Let $\varphi\in \mbox{PSH}(X,\beta)$, and let $\chi\in \mathcal{W}^-$. Then $\varphi \in \mathcal{E}_{\chi}(X,\beta)$ if and only if   there exists a decreasing sequence $\varphi_k\in \mathcal{E}_{\chi}(X,\beta) \searrow \varphi$, such that $$\sup_kE_{\chi}(\varphi_k)=\sup_k \int_X -\chi(\varphi_k-\rho) (\beta+dd^c \varphi_k)^n \textless +\infty$$
		\end{prop}
		\begin{proof}
			One side of the argument is straightforward, as we can take $\varphi_k=\varphi^{(k)}=\max\{\varphi,\rho-k\}$. To establish the other side, assume that such a decreasing sequence exists. We will now demonstrate that $\int_X -\chi(\varphi^{(k)}-\rho)(\beta+dd^c\varphi^{(k)})^n$ is uniformly bounded. Fix $j$. We know that $\varphi_k^{(j)}:=\max\{ \varphi_k,\rho-j\}$ decreases to $\varphi^{(j)}$. Therefore, using Bedford-Taylor monotone convergence theorem \cite{BT82}, we obtain
			
			\[\chi(\varphi_k^{(j)}-\rho)(\beta+dd^c \varphi_k^{(j)})^n \rightarrow \chi(\varphi^{(j)}-\rho)(\beta+dd^c\varphi^{(j)})^n \text{ weakly as } k \rightarrow \infty.\]
			
			Then we have
			\[\int_X -\chi(\varphi^{(j)}-\rho)(\beta+dd^c\varphi^{(j)})^n=\lim_{k\rightarrow \infty}\int_X -\chi(\varphi_k^{(j)}-\rho)(\beta+dd^c\varphi_k^{(j)})^n.\]
			
			Furthermore, by the proof of Proposition \ref{prop: chi energy},
			
			\[\int_X -\chi(\varphi_k^{(j)}-\rho)(\beta+dd^c\varphi_k^{(j)})^n \leq  \int_X -\chi(\varphi_k-\rho)(\beta+dd^c\varphi_k)^n,\  \forall j\in \mathbb{N}^*.\]
			
			According to the assumption, the right-hand side of the above inequality is uniformly bounded. Hence, the proof is complete.
			
		\end{proof}

		%With the above lemmas, we can get following corollary
		\begin{prop}\label{prop: criterion chi ene}
			Consider $\chi\in \mathcal{W}^-$. Let $\varphi_j,\varphi \in \mbox{PSH}(X,\beta) $ such that  $ \varphi_j \rightarrow \varphi$ in $L^1(X,\omega^n)$ as $j\rightarrow +\infty$. If the sequence $\{\varphi_j \in \mathcal{E}_{\chi}(X,\beta)\}$ has uniformly bounded $\chi$-energy, then $\varphi \in \mathcal{E}_{\chi}(X,\beta).$
				\end{prop}
		\begin{proof}
			Since the sequence $\{\varphi_j\}$ is uniformly bounded in $L^1(X,\omega^n)$, we can infer that they are uniformly bounded from above. Without loss of generality, we can assume $\varphi_j, \varphi \leq \rho$. Let us consider the function 
			
			\[\Phi_j^*:=\left(\sup_{k\geq j}\varphi_k\right)^*\in \mbox{PSH}(X,\beta).\]
			
			Since $\varphi_j\leq \Phi_j^*\leq \rho$, by Remark \ref{rem: chi ene ineq}, the sequence $\{\Phi_j^*\in \mathcal{E}_{\chi}(X,\beta)\searrow \varphi\}$ has uniformly bounded $\chi$-energy. Consequently, by Proposition \ref{prop:char echi}, we can conclude that $\varphi\in \mathcal{E}_{\chi}(X,\beta)$. 
				\end{proof}
		\begin{rem}
		Proposition \ref{prop: fund inequ} and Proposition \ref{prop: criterion chi ene} imply that the $\chi$-energy function on $\mbox{PSH}(X,\beta)$ is lower-semi-continuous, as follows:
			\[ E_\chi(\varphi)\leq 2^n\liminf_{\psi \rightarrow \varphi}E_\chi(\psi),\]
			where $\psi \rightarrow \varphi$ in $L^1(X,\omega^n)$.
			\end{rem}
		
		\section{Range of the  CMA operator on the class $\mathcal E^p(X,\beta)$}\label{sect: range}
		In this section, we investigate the range of the complex Monge-Amp\`ere operator $\langle(\beta+dd^c\cdot)^n\rangle$. According to \cite[Theorem 1]{LWZ23}, without loss of generality, we can assume that the bounded potential $\rho$ satisfies $(\beta+dd^c \rho)^n=\omega^n$.

		\begin{lem}[{\cite[Proposition 2.7]{LWZ23}}]\label{lem: volume cap esti}
			There exist uniform positive constants $C$ and $\alpha$ such that for any Borel set $E$, we have
			$$Vol_{\omega}(E)\leq C\exp(-\alpha Cap_{\beta}(E)^{-\frac{1}{n}}).$$
		\end{lem}
		
		\begin{lem}[{\cite[Lemma 3.11]{LWZ23}}]\label{lem: cap est upper level set}
		There exist a uniform constant $C'>0$ such that for any $v\in \mbox{PSH}(X,\beta+dd^c\rho)$ with $\sup_X v=0$ and $t>0$, we have
		$$Cap_{\beta}(\{v< -t\})\leq \frac{C'}{t}.$$
		\end{lem}
		\begin{rem}
			The above Lemma also holds for any $v\in \mbox{PSH}(X,\beta+dd^c\rho)$ with the normalization condition $\sup_X v=-A$, where $A>0$ is a positive constant.
		\end{rem}
Let $\mathcal{W}^+$ be the class of all smooth concave functions $\chi:\mathbb{R}^-\rightarrow \mathbb{R}^-$ such that $\chi(0)=0$ and $\chi(-\infty)=-\infty$. For $M>0$, we define $\mathcal{W}^+_M$ as follows:
$$\mathcal{W}^+_M:=\{\chi \in \mathcal{W}^+:0\leq|t\chi'(t)|\leq M|\chi(t)|, \forall \ t\in \mathbb{R}^- \}.$$
		\begin{prop}\label{prop: cap estima upper level set 2}
Let $\chi \in \mathcal{W}^-_h\cup \mathcal{W}^+_M$. Then there exists a uniform constant $C>0$ such that for any $v\in \mathcal{E}_\chi(X,\beta+dd^c \rho)$ with $\sup_X v=-1$ and $t>0$, we have
$$Cap_{\beta}(\{v< -t\})\leq \frac{CE_{\chi}(\rho+v)}{-t\chi(-t)}.$$
In particular, if $\chi(t)=-(-t)^p$ for $t\in \mathbb{R}^-$ and $p\geq1$, then
$$Cap_{\beta}(\{v< -t\})\leq \frac{CE_{\chi}(\rho+v)}{t^{p+1}}.$$
		\end{prop}
		
		\begin{proof}
Fix an arbitrary $u \in \mbox{PSH}(X, \beta + dd^c \rho),$ with $-1 \leq u \leq 0$. If $\chi \in \mathcal{W}^-$, then $-\chi(-2t) \leq -2\chi(-t)$ for $t \in \mathbb{R}^+$. Furthermore, if $\chi \in \mathcal{W}_M^+$, then $-\chi(-2t) \leq -2^M \chi(-t)$ for $t \in \mathbb{R}^+$. Therefore, it suffices to prove that 
$$\int_{\{v < -2t\}}(\beta+dd^c\rho+dd^cu)^n \leq \frac{CE_{\chi}(\rho+v)}{-t\chi(-t)}.$$			
To show this, we observe that 
%			Since $\{v\textless -2t \} \subseteq \{\frac{v}{t} \textless -1+u\} \subseteq \{v \textless -t\}$, by the comparison principle (Theorem \ref{thm: comparison prin}), 
			\begin{align*}
				&\int_{\{v\textless -2t\}}(\beta+dd^c\rho+dd^cu)^n\\
				&\leq \int_{\{\frac{v}{t} \textless -1+u\}}(\beta+dd^c\rho+dd^cu)^n\\
				&\leq \int_{\{\frac{v}{t} \textless -1+u\}}\left(\beta+dd^c\rho+dd^c\frac{v}{t}\right)^n\\
				&\leq \int_{\{v \textless -t\}}\left(\beta+dd^c\rho+dd^c\frac{v}{t}\right)^n\\
				&\leq \int_{\{v \textless -t\}}[\beta+dd^c\rho+\frac{1}{t}(\beta+dd^c\rho+dd^cv)]^n\\
				&=\int_{\{v \textless  -t\}}(\beta+dd^c\rho)^n+\frac{1}{t}\sum_{j=0}^{n-1} \frac{1}{t^{n-j-1}}\binom{n}{j}\int_{\{v \textless -t\}}(\beta+dd^c\rho)^j\wedge(\beta+dd^c\rho+dd^cv)^{n-j}.
			\end{align*}
			By Lemma \ref{lem: volume cap esti} and Lemma \ref{lem: cap est upper level set}, we have 
			\begin{align*}
				\int_{\{v\textless-t\}}(\beta+dd^c\rho)^n&=Vol_{\omega}(\{v\textless -t\})\\
				&\leq C \exp(-\alpha Cap_{\beta}(\{v\textless-t\})^{-\frac{1}{n}})\\ 
				&\leq C \exp\left(-\alpha \left(\frac{t}{C'}\right)^{\frac{1}{n}}\right)=o\left(\frac{1}{-t\chi(-t)}\right).
			\end{align*}

Next, we consider $\int_{\{v < -t\}}(\beta+dd^c\rho)^j\wedge(\beta+dd^c\rho+dd^cv)^{n-j}$, for $j = 0, 1, \ldots, n-1$. 
We observe that 
			$$\int_{\{v \textless -t\}}(\beta+dd^c\rho)^j\wedge(\beta+dd^c\rho+dd^cv)^{n-j}\leq \frac{1}{-\chi(-t)}\int_X-\chi(v)(\beta+dd^c\rho)^j\wedge(\beta+dd^c\rho+dd^cv)^{n-j}.$$
			By Proposition \ref{prop: mix chi energy ineq},
			\begin{align*}
				\int_X-\chi(v)(\beta+dd^c\rho)^j\wedge(\beta+dd^c\rho+dd^cv)^{n-j}\leq\int_X -\chi(v)(\beta+dd^c\rho+dd^cv)^n=E_{\chi}(\rho+v).
			\end{align*}
To summarize, we have$$\int_{\{v\textless -2t\}}(\beta+dd^c\rho+dd^cu)^n \leq \frac{CE_{\chi}(\rho+v)}{-t\chi(-t)}. $$
			%This is equivalent to what we want.
		\end{proof}
		%\begin{rem}
		%We only use the above theorem in the case that $\chi(t)=-(-t)^p$, $t\in \mathbb{R}^-$, $p\geq1$.
		%\end{rem}
	We have the following $L^1$-estimate of $v \in \mathcal{E}^1(X,\beta+dd^c\rho)$ with respect to the measure $d\mu$, for $\mu \in \mathcal{H}(A,a,\beta)$.
		\begin{lem}\label{lem: l1 est of v}
			Let $\mu \in \mathcal{H}(A,a=\frac{1}{2}+\varepsilon,\beta)$, $\varepsilon>0$. Then there exists a uniform constant $C>0$, for any $v \in \mathcal{E}^1(X,\beta+dd^c\rho)$ with $\sup_X v=-1$, we have
			\[
			\int_X -v d\mu \leq C'E_1(\rho+v)^{\frac{1}{2}+\varepsilon}.
			\]
				\end{lem}
		\begin{proof}

			Since $\mu \in \mathcal{H}(A,a=\frac{1}{2}+\varepsilon,\beta)$, by the Fubini theorem and Proposition \ref{prop: cap estima upper level set 2}, we have
			\begin{align*}
				\int_X -v d\mu&=\mu(X)+\int_1^{+\infty}d\mu(\{v \textless -t\})dt\\
				&\leq \mu(X)+\int_1^{+\infty}ACap_{\beta}(\{v \textless -t\})^{\frac{1}{2}+\varepsilon}dt\\
				&\leq \mu(X)+\int_1^{+\infty}A\left(\frac{CE_1(\rho+v)}{t^2}\right)^{\frac{1}{2}+\varepsilon}\\
				&\leq C'E_1(\rho+v)^{\frac{1}{2}+\varepsilon}.
			\end{align*}
				\end{proof}
		
		\begin{prop}[cf. {\cite[Proposition 2.8]{GZ07}}]\label{prop: energy conv} 
		Let $\varphi_j,\varphi \in \mathcal{E}^1(X,\beta)$, such that $\varphi_j \rightarrow \varphi$ in $L^1(X,\omega^n)$ and $\lim_{j\rightarrow\infty}\int_X|\varphi_j-\varphi|\langle(\beta+dd^c \varphi_j)^n\rangle=0$. Then $\langle(\beta+dd^c\varphi_j)^n\rangle \rightarrow \langle(\beta+dd^c\varphi)^n\rangle$.
			\end{prop}
		
		\begin{proof}
		
	We replace the corresponding functions $\Phi_j, \Phi^K_j,\varphi^K, j,K\in \mathbb N^*$ in \cite[Proposition 2.8]{GZ07} by the following functions:
	\[
	\Phi_j:=\max\{\varphi_j, \varphi-\frac{1}{j}\},
	\]
	\[
	\Phi_j^{(K)}:=\max\{ \Phi_j, \rho-K\},
	\]
	\[
	\varphi^{(K)}:=\max\{\varphi, \rho-K\}.
	\]
		Then the same proof in \cite[Proposition 2.8]{GZ07} works, and we omit the details here.
	\end{proof}
Let $\mu\in \mathcal{H}(A,a,\beta)$ be such that $\mu(X)=\int_X\beta^n=1$. Let $\{U_j\}$ be a covering of $X$ by coordinate balls $U_j$, and let $\{\theta_j\}$ be a partition of unity with respect to the covering $\{U_j\}$. Set $\mu_j= c_j \sum_{i=1}^N (\theta_i \mu_{|U_i})*\tau_{\varepsilon_j}$, where $\tau_{\varepsilon_j}$ are standard convolution kernels, and $\{c_j\}$ is a sequence that tends to $1$, such that $\mu_j(X)=\int_X \beta^n=1$. By \cite[Theorem 1]{LWZ23}, there exist $v_j\in \mbox{PSH}(X,\beta+dd^c\rho)\cap L^\infty(X)$ such that $(\beta+dd^c\rho+dd^cv_j)^n=\mu_j$ and $\sup_X v_j=-1$. By the Hartogs lemma, we may assume that $v_j$ converges to some $v\in \mbox{PSH}(X,\beta+dd^c\rho)$ in $L^1(X,\omega^n)$.

		\begin{lem}\label{lem: pf of main 1-lem 1}
			There exists a uniform constant $C>1$ such that for all $j\in \mathbb N^*$,
			$$\int_X -v_j(\beta+dd^c\rho+dd^cv_j)^n \leq C \int_X -v_j d\mu.$$
			Moreover, if $\mu \in \mathcal{H}(A,\frac{1}{2}+\varepsilon,\beta)$ for some $\varepsilon>0$, we have $v \in \mathcal{E}^1(X,\beta+dd^c\rho)$.
				\end{lem} 
		\begin{proof}

			Set $\mu_{j,i}=(\theta_i\mu_{|U_i})*\tau_{\varepsilon_j}$, then 
			$$\int_X -v_j(\beta+dd^c\rho+dd^cv_j)^n=\sum_{i=1}^N c_j\int_X (-v_j)d\mu_{j,i}.$$
			By definition, we have $$\int_X (-v_j)d\mu_{j,i}=\int_{U_i}-v_j*\tau_{\varepsilon_j}\theta_id\mu \leq \int_{U_i}-v_j*\tau_{\varepsilon_j}d\mu.$$
			Write $\beta=dd^c\psi_i$ on each $U_i$. Since $v_j+\rho+\psi_i$ is psh on $U_i$, 
			\begin{align*}
				-v_j*\tau_{\varepsilon_j}=-(v_j+\rho+\psi_i)*\tau_{\varepsilon_j}+(\rho+\psi_i)*\tau_{\varepsilon_j} \leq -(v_j+\rho+\psi_i)+(\rho+\psi_i)*\tau_{\varepsilon_j}.
			\end{align*}
			It follows that 
			$$\int_X (-v_j)d\mu_{j,i} \leq \int_{U_i} -v_jd\mu+o(1),$$
			where $o(1)=\int_{U_i} (\rho+\psi_i)*\tau_{\varepsilon_j}-(\rho+\psi_i)$, and thus we have 
			$$\int_X -v_j(\beta+dd^c\rho+dd^cv_j)^n \leq C \int_X -v_j d\mu.$$
			Furthermore, if $\mu \in \mathcal{H}(A,\frac{1}{2}+\varepsilon,\beta)$ for some $\varepsilon > 0$, by Lemma \ref{lem: l1 est of v}, we have
			$$\int_X-v_jd\mu \leq C'\left(\int_X-v_j(\beta+dd^c\rho+dd^cv_j)^n\right)^{\frac{1}{2}+\varepsilon}.$$
			This implies that $E_1(\rho+v_j) \leq CC' E_1(\rho+v_j)^{\frac{1}{2}+\varepsilon}$, i.e. $E_1(\rho+v_j)$ is uniformly bounded. By Proposition \ref{prop: criterion chi ene}, we have $v\in \mathcal{E}^1(X,\beta+dd^c\rho)$. The proof of Lemma \ref{lem: pf of main 1-lem 1} is complete.
			
		\end{proof}
		\begin{lem}\label{lem: dmu conv}
			For $\mu \in \mathcal{H}(A,\frac{1}{2}+\varepsilon,\beta)$, $\varepsilon > 0$, we have 
			$$\lim_{j \rightarrow \infty}\int_X|v_j-v|d\mu= 0$$
			and
			$$\lim_{j \rightarrow \infty}\int_X|v_j-v|d\mu_j= 0.$$
		\end{lem}
		\begin{proof}
			We will first prove that
			$$\lim_{j \rightarrow \infty}\int_X v_j d\mu = \int_X v d\mu.$$
			
			Since 
			\begin{align*}
				\int_X(-v_j)^{\delta}d\mu&=\mu(X)+\delta\int_{1}^{+\infty}t^{\delta-1}d\mu(\{v_j\textless-t\})dt\\
				&\leq 1+A\delta \int_1^{+\infty}t^{\delta-1}Cap_{\beta}(\{v_j\textless-t\})^{\frac{1}{2}+\varepsilon}dt\\
				&\leq 1+A\delta \int_1^{+\infty}t^{\delta-1}\frac{CE_1(\rho+v_j)}{t^{1+2\varepsilon}}dt.
			\end{align*}
			Hence, the sequence $\{v_j\}$ has uniform $L^{\delta}(X,d\mu)$-norm for some $\delta := 1+\varepsilon > 1$.
			Up to a subsequence, we may assume that the sequence $\{v_j\}$ converges weakly to $u \in L^{\delta}(X,d\mu)$. 
			By the Mazur theorem (see \cite[Corollary 3.8]{Bre11}), there exist $0\leq \varepsilon_{j,0},\varepsilon_{j,1},\cdots,\varepsilon_{j,t_j}\leq1$ such that 
			$$\varepsilon_{j,0}+\varepsilon_{j,1}+\cdots+\varepsilon_{j,t_j}=1 \text{ for } j\in \mathbb N^* \text{ and }\|(\varepsilon_{j,0}v_j+\varepsilon_{j,1}v_{j+1}+\cdots+\varepsilon_{j,t_j}v_{t_j})-u\|_{L^{\delta}(d\mu)} \rightarrow 0 \text{ as } j \rightarrow \infty.$$ 
			Thus, we may assume 
			$$\lim_{j\rightarrow \infty} (\varepsilon_{j,0}v_j+\varepsilon_{j,1}v_{j+1}+\cdots+\varepsilon_{j,t_j}v_{t_j}+\rho) =u+\rho, \ \  d\mu\text{-a.e.}$$
			Therefore, 
			$$\lim_{j\rightarrow \infty} \sup_{l\geq j}(\varepsilon_{l,0}v_l+\varepsilon_{l,1}v_{l+1}+\cdots+\varepsilon_{l,t_l}v_{t_l}+\rho)=u+\rho, \ \  d\mu\text{-a.e.}$$ 
			By \cite[Proposition 5.1]{BT82}, the (Bedford-Taylor) capacity the negligible set 
			$$\left\{\left(\sup_{l\geq j} (\varepsilon_{l,0}v_l+\varepsilon_{l,1}v_{l+1}+\cdots+\varepsilon_{l,t_l}v_{t_l}+\rho)\right)^*\textgreater\sup_{l\geq j} (\varepsilon_{l,0}v_l+\varepsilon_{l,1}v_{l+1}+\cdots+\varepsilon_{l,t_l}v_{t_l}+\rho)\right\}$$
			is 0. Therefore, we have
			$$\mu\left(\left\{\left(\sup_{l\geq j} (\varepsilon_{l,0}v_l+\varepsilon_{l,1}v_{l+1}+\cdots+\varepsilon_{l,t_l}v_{t_l}+\rho)\right)^*\textgreater\sup_{l\geq j} (\varepsilon_{l,0}v_l+\varepsilon_{l,1}v_{l+1}+\cdots+\varepsilon_{l,t_l}v_{t_l}+\rho)\right\}\right)=0.$$  
			It follows that
			$$\lim_{j\rightarrow \infty} \left(\sup_{l\geq j} (\varepsilon_{l,0}v_l+\varepsilon_{l,1}v_{l+1}+\cdots+\varepsilon_{l,t_l}v_{t_l}+\rho)\right)^*=u+\rho \ \  d\mu \text{-almost everywhere}.$$
			Considering that $v_j$ converges to $v$ in $L^1(X,\omega^n)$, we deduce that the sequence $$\left\{\left(\sup_{l\geq j}(\varepsilon_{l,0}v_l+\varepsilon_{l,1}v_{l+1}+\cdots+\varepsilon_{l,t_l}v_{t_l}+\rho)\right)^*\right\}$$ decreases to $v+\rho$, which belongs to $L^{\delta}(X,d \mu)$. By the Dominated Convergence Theorem, we have
			$$\lim_{j\rightarrow \infty}\int_X\left(\sup_{l\geq j} (\varepsilon_{l,0}v_l+\varepsilon_{l,1}v_{l+1}+\cdots+\varepsilon_{l,t_l}v_{t_l}+\rho)\right)^*d\mu=\int_X(u+\rho)d\mu.$$
			Additionally, according to the Monotone Convergence Theorem, we obtain
			$$\lim_{j\rightarrow \infty}\int_X\left(\sup_{l\geq j} (\varepsilon_{l,0}v_l+\varepsilon_{l,1}v_{l+1}+\cdots+\varepsilon_{l,t_l}v_{t_l}+\rho)\right)^*d\mu=\int_X(v+\rho)d\mu.$$
			Thus, we conclude that
			$$\int_Xvd\mu=\int_Xud\mu.$$
			Given that $v_j$ converges to $u$ in $L^{\delta}(X,d\mu)$, we have
			$$\lim_{j\rightarrow \infty}\int_Xv_jd\mu=\int_Xud\mu,$$
			which implies
			$$\lim_{j\rightarrow \infty}\int_Xv_jd\mu=\int_Xvd\mu.$$
			
			Next, we aim to prove that
			$$\lim_{j \rightarrow \infty}\int_X|v_j-v|d\mu =0.$$
			The proof follows a similar approach as in  \cite[Lemma 4.4]{GZ07}. Let $V_j:=(\sup_{k\geq j}v_k+\rho)^*$, and note that $V_j$ decreases to $v+\rho$ on $X$. Since 
			\begin{align*}
				|v_j-v|&=|(V_j-v-\rho)-(V_j-v_j-\rho)|\\
				&\leq (V_j-v-\rho)+(V_j-v_j-\rho) = 2(V_j-v-\rho)+(v-v_j),
			\end{align*}
			by the Monotone Convergence Theorem, we can write
			\begin{align*}
				\int_X  |v_j-v|d\mu\leq 2\int_X(V_j-v-\rho)d\mu+\int_X(v-v_j)d\mu\rightarrow 0.
			\end{align*}
			
			To complete the proof of Lemma \ref{lem: dmu conv}, it remains to show that 
			$$\int_{U_i}|v_j-v|d\mu_{j,i}\xrightarrow{j\rightarrow \infty} 0.$$
			
			By definition, we have $\int_{U_i}|v_j-v|d\mu_{j,i}=\int_{U_i}\theta_i|v_j-v|*\tau_{\varepsilon_j}d \mu\leq \int_{U_i}|v_j-v|*\tau_{\varepsilon_j}d \mu$. Let $V_j':=(\sup_{k\geq j}(v_k+\rho+\psi_i))^*$, where $V_j'$ is a plurisubharmonic (psh) function that decreases to $v+\rho+\psi_i$, and $V_j'\geq v_j+\rho+\psi_i$. Now, consider the following:
			\begin{align*}
				|v_j-v|&=|(V_j'-(v_j+\rho+\psi_i))-(V_j'-(v+\rho+\psi_i))|\\
				&\leq (V_j'-(v_j+\rho+\psi_i))+(V_j'-(v+\rho+\psi_i)).
			\end{align*}
			This implies that 
			\begin{align*}
				\int_{U_i}|v_j-v|*\tau_{\varepsilon_j}d\mu
				&\leq  \int_{U_i}(V_j'-(v_j+\rho+\psi_i))*\tau_{\varepsilon_j}d\mu+\int_{U_i}(V_j'-(v+\rho+\psi_i))*\tau_{\varepsilon_j}d\mu\\
				&=\int_{U_i}2V_j'*\tau_{\varepsilon_j}d\mu -\int_{U_i}(v+\rho+\psi_i)*\tau_{\varepsilon_j}d\mu-\int_{U_i}(v_j+\rho+\psi_i))*\tau_{\varepsilon_j}d\mu\\
				&\leq \int_{U_i}2(V_j'*\tau_{\varepsilon_j}-(v+\rho+\psi_i))d\mu+\int_{U_i}((v+\rho+\psi_i)-(v_j+\rho+\psi_i))d\mu. 
			\end{align*}
			The first term tends to 0 by the monotone convergence theorem, and the second term tends to 0 by $\int_X|v_j-v|d\mu \rightarrow0$ as $j \rightarrow \infty$.
							\end{proof}

	The main result of this section is as follows:
		\begin{thm}[=Theorem \ref{thm: main 1-1}]\label{thm: main 1}
		Let $p\geq1$ and $a>\frac{p}{p+1}$. Assume that $\mu\in \mathcal{H}(A,a,\beta)$ satisfies $\mu(X)=\int_X\beta^n=1$. Then there exists a unique $\varphi\in \mathcal{E}^p(X,\beta)$ such that
		$$\langle(\beta+dd^c\varphi)^n\rangle=\mu, \  \sup_X\varphi=0.$$
			\end{thm}
		
		\begin{rem}
		The above theorem is a generalization of the corresponding result in \cite[Theorem 4.2, Proposition 5.3]{GZ07}, where $X$ is assumed to be K\"ahler and $\beta$ is assumed to be a K\"ahler metric on $X$. 
			\end{rem}
		
		Note that if $a > \frac{p}{p+1}$, then there exists an $\varepsilon > 0$ such that $a > \frac{p+\varepsilon}{p+\varepsilon+1}$. Thus, the proof of Theorem \ref{thm: main 1} can be reduced to the proof of the following theorem:
			\begin{thm}\label{thm: main 1 weak}
				For any $\mu \in \mathcal{H}(A,a,\beta)$, where $a > \frac{p}{p+1}$ and $p \geq 1$, such that $\mu(X) = \int_X\beta^n = 1$, there exists a unique $\varphi \in \cap_{\{p' < p\}}\mathcal{E}^{p'}(X,\beta)$ such that
				$$\langle(\beta+dd^c\varphi)^n\rangle = \mu, \quad \sup_X \varphi = 0.$$
		\end{thm}
		\begin{proof}
			Since $p\geq1,$ we have $a>\frac{p}{p+1}\geq\frac{1}{2}$. By Proposition \ref{prop: energy conv} and Lemma \ref{lem: dmu conv}, there exists $v \in \mathcal{E}^1(X,\beta+dd^c\rho)$ such that
			$$\langle(\beta+dd^c\rho+dd^cv)^n\rangle=\mu, \  \sup_Xv=-1.$$
			
			In the following, we show that $v\in\mathcal{E}^{p'}(X,\beta+dd^c\rho)$ for any $p'< p$, i.e.
				$$\int_X(-v)^{p'}\langle(\beta+dd^c\rho+dd^cv)^n\rangle< +\infty.$$
			
			Set $p_1=1$, $p_{k+1}=\frac{p}{p+1}(1+p_k)$, $k\in \mathbb{N}^*$. It is clear that $\{p_k\}\nearrow p$. We prove by induction that $v\in \mathcal{E}^{p_k}(X,\beta+dd^c\rho)$, $k\in \mathbb{N}^*$.
			
			For $p_1=1$, it is obvious since $v \in \mathcal{E}^1(X,\beta+dd^c\rho)$.

Assume now $v \in \mathcal{E}^{p_k}(X,\beta+dd^c\rho)$. By the Fubini Theorem, we get that
$$\int_X(-v)^{p_{k+1}}\langle(\beta+dd^c\rho+dd^cv)^n\rangle\leq 1+p_{k+1}\int_1^{+\infty}t^{p_{k+1}-1}\mu(\{v<-t\})dt.$$
Since $\mu \in \mathcal{H}(A,a,\beta)$, by Proposition \ref{prop: cap estima upper level set 2}:
$$d\mu(\{v<-t\})\leq ACap_{\beta}(\{v<-t\})^{a}\leq A\left(\frac{CE_{p_k}(\rho+v)}{t^{1+p_k}}\right)^{a}.$$
It follows that 
\begin{align*}
	\int_X(-v)^{p_{k+1}}\langle(\beta+dd^c\rho+dd^cv)^n\rangle&\leq 1+p_{k+1}\int_1^{+\infty}t^{p_{k+1}-1}\mu(\{v<-t\})dt\\
	&\leq 1+p_{k+1}\int_1^{+\infty}t^{p_{k+1}-1}A\left(\frac{CE_{p_k}(\rho+v)}{t^{1+p_k}}\right)^{a}dt\\
	&=1+p_{k+1}A(CE_{p_k}(\rho+v))^a\int_1^{+\infty}\frac{dt}{t^{a(1+p_k)-p_{k+1}+1}}\\
	&<+\infty.
\end{align*}
The last inequality is due to the fact that $a(1+p_k)-p_{k+1}+1>\frac{p}{p+1}(1+p_k)-(p_{k+1}-1)=1$. This completes the existence part of the proof.
		
			The argument in  \cite[\S 3.3]{BEGZ10} works word by word here to prove the uniqueness part of Theorem \ref{thm: main 1 weak}.
			
		\end{proof}

		As a special case of Theorem \ref{thm: main 1}, say  $a>1$, we can get the $L^\infty$-estimate of the solution $\varphi$ by the method in \cite[\S 2]{EGZ09}, see also \cite[\S 3.4]{LWZ23}.
		\begin{cor}[=Corollary \ref{cor: 1}]\label{cor: a>1}
			
			Let  $a>1$, and let $\mu\in \mathcal{H}(A,a,\beta)$ such that   $\mu(X)=\int_X\beta^n=1$. Then, there exists a unique $\varphi\in \mbox{PSH}(X,\beta)\cap L^{\infty}(X)$ such that 
			$$(\beta+dd^c\varphi)^n=\mu \quad \text{and} \quad \sup_X\varphi=0.$$

		\end{cor}
		
		\begin{rem}
			The above corollary is a generalization of the result in \cite[Theorem 2.1]{EGZ09}, where  $X$ is assumed to be K\"ahler and $\beta$ is assumed to be semi-positive.
		\end{rem}

		\section{Characterization of the range of the  CMA operator on the class  $\mathcal{E}(X,\beta)$}\label{sect: char}
In this section, we provide a complete characterization of the range of the complex non-pluripolar Monge-Amp\`ere operator on the class $\mathcal E(X,\beta)$.

	\begin{thm}[=Theorem \ref{thm: main 1-2}]\label{thm: main 2}
		Let $(X,\omega)$ be a compact Hermitian manifold of complex dimension $n$. Let $\{\beta\}\in H^{1,1}(X,\mathbb R)$ be a real $(1,1)$-class with a smooth representative $\beta$. Assume that there is a continuous function $\rho$ such that $\beta+dd^c\rho\geq 0$ in the sense of currents. Then the following two conditions are equivalent:
			\begin{itemize}	
			\item [(1)] For every non-pluripolar Randon measure $\mu$ satisfying $\mu(X)=\int_X \beta^n=1$, there exists $\varphi \in \mathcal{E}(X,\beta)$ such that
			$$\langle(\beta+dd^c\varphi)^n\rangle=\mu.$$
			\item [(2)]The capacity $Cap_{\beta}(\cdot)$ can characterize pluripolar sets in the following sense: for any Borel set $E$, $Cap_{\beta}(E)=0$ if and only if $E$ is pluripolar.
			\end{itemize}
	\end{thm}
	
\begin{proof}[Proof of (1)$\implies$(2)]
	Suppose to the contrary, we assume that $E$ satisfies $Cap_{\beta}(E)=0$ but $E$ is not pluripolar. By \cite[Theorem 6.9, Theorem 8.2]{BT82}, there are concentric coordinate balls $U\Subset U'$ and a function $u\in\mbox {PSH}(U')\cap L^\infty(U')$ such that $\int_{E\cap U}(dd^cu)^n>0$.
	
	Let $\chi$ be a cut-off function such that $\chi=1$ on $U$ and $supp(\chi)\subset U'$. Let $c>0$ be a positive number such that $\int_X c\chi(dd^cu)^n=\int_X\beta^n$. Set $\mu=c\chi(dd^cu)^n$, then $\mu$ is a non-pluripolar measure. By (1), there exists $v\in \mathcal{E}(X,\beta+dd^c\rho)$ such that $\langle(\beta+dd^c\rho+dd^cv)^n\rangle=\mu$ and $\sup_Xv=0$. Thus, there is a large integer $k$ such that $$\int_{\{v>-k\}\cap E}(\beta+dd^c\rho+dd^cv^{(k)})^n>0,$$ where $v^{(k)}:=\max\{v,-k\}$.
	
	Since $$(\beta+dd^c\rho+dd^cv^{(k)})^n\leq k^n\left(\beta+dd^c\rho+dd^c\frac{v^{(k)}}{k}\right)^n,$$ we have $$\left(\beta+dd^c\rho+dd^c\frac{v^{(k)}}{k}\right)^n(E)>0.$$ This contradicts $Cap_{\beta}(E)=0$.
	\end{proof}
		\begin{rem}
		We would like to emphasize that, in order for the above proof to work, it is sufficient for $\rho$ to be bounded.
		\end{rem}
To proceed, we require the following lemma:
		\begin{lem}\label{lem: npp measure}
	Let $\mu$ be a non-pluripolar positive Radon measure on $X$, and let $\rho\in \mbox{PSH}(X,\beta)\cap \mathcal C(X)$. Assume that $Cap_{\beta}(\cdot)$ can characterize pluripolar sets. Then there exists $u\in \mbox{PSH}(X,\beta+dd^c\rho)\cap L^{\infty}(X)$ and $0\leq f\in L^1((\beta+dd^c\rho+dd^c u)^n)$ such that $\mu=f(\beta+dd^c\rho+dd^cu)^n.$
	\end{lem}
		\begin{proof}
	The proof is almost the same as \cite[Lemma 4.5]{GZ07}, with the exception of a new ingredient: Theorem \ref{thm: main 1 weak}. Recalling the definition:
				$$\mathcal{H}(1,1,\beta):=\{\nu\in \mathcal{M}_X: \nu \leq Cap_{\beta}\}.$$
				Since $\rho$ is continuous, according to \cite[Proposition 1.6]{BBGZ13}, we know that $\mathcal{H}(1,1,\beta)$ is compact in the weak topology. From a generalization of the Radon-Nikodym theorem (see \cite{Ra69}), we have:
			
		\begin{center}
				$\mu=g\nu+\sigma$, where $\nu \in \mathcal{H}(1,1,\beta)$, $0\leq g \in L^1(\nu)$, and $\sigma \perp \mathcal{H}(1,1,\beta)$.
			\end{center}
				It is clear that all measures of the form $(\beta+dd^c\rho+dd^cv)^n$, where $v \in \mbox{PSH}(X,\beta+dd^c\rho)$ and $0\leq v \leq 1$, are contained in $\mathcal{H}(1,1,\beta)$. Additionally, $Cap_{\beta}(\cdot)$ can characterize pluripolar sets. Therefore, we see that $\sigma$ is carried by a pluripolar set. However, since $\sigma=\mu-g\nu$ is non-pluripolar, this implies that $\sigma=0$.
				Using Theorem \ref{thm: main 1 weak}, there exists $v \in \mathcal{E}^1(X,\beta+dd^c\rho) $ with $\sup_Xv=0$, such that $\langle(\beta+dd^c\rho+dd^cv)^n\rangle=\nu$. We set $u:=e^v$, then $u\in \mbox{PSH}(\beta+dd^c\rho)\cap L^\infty(X)$ and we have:
				$$\beta+dd^c\rho+dd^ce^v=e^v(\beta+dd^c\rho+dd^cv)+(1-e^v)(\beta+dd^c\rho)+e^vdv\wedge d^cv$$ $$\geq e^v(\beta+dd^c\rho+dd^cv)\geq0.$$
				Therefore, we have:
				$$ e^{nv}\nu= e^{nv}\langle(\beta+dd^c\rho+dd^cv)^n\rangle\leq (\beta+dd^c\rho+dd^cu)^n,$$
				which implies that $\mu$ is absolutely continuous with respect to $(\beta+dd^c\rho+dd^cu)^n$. This completes our proof.
				\end{proof}
		
		\begin{proof}[Proof of (2)$\implies$(1)]
			We follow the strategy outlined in \cite[Theorem 4.6]{GZ07}.  Let's assume that $\mu$ is a non-pluripolar Radon measure such that $\mu(X)=\int_X\beta^n$. By using Lemma \ref{lem: npp measure}, we can write $\mu=f\nu$, where $f$ is a non-negative function in $L^1(\nu)$, and $\nu=(\beta+dd^c\rho+dd^cu)^n$ for some $0\leq u \leq1$, where $u\in \mbox{PSH}(X,\beta+dd^c\rho)$.  We define $\mu_j:=c_j\min\{f,j\}\nu$, with $c_j\searrow 1$ such that $\mu_j(X)=\int_X\beta^n$. It can be easily seen that $\mu_j\leq 2\min\{f,j\}\nu\leq 2f\nu$ and $\mu_j\leq 2j\nu$ for sufficiently large $j$, hence $\mu_j\in \mathcal H(2j,1,\beta)$. By applying Theorem \ref{thm: main 1}, we can find a sequence of $v_j \in \mathcal{E}^1(X,\beta+dd^c\rho)$, with $\sup_X v_j=0$, such that $\langle(\beta+dd^c\rho+dd^cv_j)^n\rangle=\mu_j$. Up to a subsequence, we can assume that $v_j\rightarrow v$ in $L^1(X,\omega^n)$ for some $v\in \mbox{PSH}(X,\beta+dd^c\rho)$ with $\sup_Xv=0$.
			
			We will now show that $v\in \mathcal{E}_{\chi}(X,\beta+dd^c\rho)$ for some $\chi\in \mathcal{W}^-$ and that $(\beta+dd^c\rho+dd^cv)^n=\mu$.
			
			 To construct the function $\chi$, let $\gamma:\mathbb{R}^{+} \rightarrow \mathbb{R}^{+}$ be a smooth, convex, and increasing function such that $\gamma(0)=\gamma'(0)=0$ and $\lim_{x\rightarrow +\infty}\gamma(x)/x=+\infty$, while still having $\gamma \circ f$ in $L^1(\nu)$ (cf. \cite{RR91} for the construction of $\gamma$). We define $\gamma^*$ as the Legendre transform of $\gamma$, where $\gamma^*(x)=\sup\{xy-\gamma(y)|y\in [0,\infty) \}\in [0,\infty)$. We then set $\tilde\chi(t):=-(\gamma^*)^{-1}(-t)$, which gives us a convex increasing function defined on $(-\infty,0]$ such that $\tilde\chi(-\infty)=-\infty$, $\tilde\chi(0)=0$, and $$-\tilde\chi(t)\cdot f(x)\leq -t+\gamma \circ f(x), \ \forall (t,x)\in (-\infty,0] \times X\label{equ:1}$$ since $xy-\gamma(y)\leq \gamma^*(x)$ for all $x,y\in [0,\infty)$. We will now modify $\tilde\chi$ to be a weight function defined on $\mathbb{R}$ as follows: let $g: x\in [0,\infty) \rightarrow g(x)\in [0,\infty)$ such that $\gamma'(g(x))=x$. This is a smooth increasing function with the properties $\gamma^*(x)=xg(x)-\gamma(g(x))$, $\lim_{x\rightarrow0}g(x)=0$, and $\lim_{x\rightarrow \infty}g(x)=\infty$. Furthermore, $(\gamma^*)'(x)=g(x)$ and $\tilde\chi'(t)=\frac{1}{(\gamma^*)'[(\gamma^*)^{-1}(-t)]}=\frac{1}{g[(\gamma^*)^{-1}(-t)]}$.
			
			Since $\lim_{x\rightarrow 0} (\gamma^*)(x)=0$ and $\lim_{x\rightarrow \infty}(\gamma^*)(x)=\infty$, we have $\lim_{t\rightarrow 0^-} (\gamma^*)^{-1}(-t)=0$ and $\lim_{t\rightarrow -\infty}(\gamma^*)^{-1}(-t)=\infty$. Therefore, $\lim_{t\rightarrow -\infty}\chi'(t)=0$ and $\lim_{t \rightarrow 0^-}\chi'(t)=\infty$.
			
			Let $0<\varepsilon<\eta$ be two sufficiently small constants. Define $$
			\chi(t)=
			\begin{cases}
				\max_{\varepsilon}\{\tilde\chi(t)-\eta,t\} &\text{if } t<0, \\
				t &\text{if } t\geq0.
			\end{cases}
			$$
			
			Note that $\chi$ is a smooth convex function on $\mathbb{R}$ with $\chi(t)=t$ on $[0,\infty)$, i.e. $\chi\in \mathcal W^-$.
			
			Then we have  \begin{align*}
				&\int_X(-\chi\circ v_j)\langle(\beta+dd^c\rho+dd^cv_j)^n\rangle\\
				&\leq 2\int_X(-\chi\circ v_j) f\nu\\
				&\leq 2\int_X-v_j(\beta+dd^c\rho+dd^cu)^n+2\int_X(\gamma \circ f+\eta f)(\beta+dd^c\rho+dd^cu)^n.
			\end{align*} 
			
			By the Chern-Levine-Nirenberg inequality (see e.g. \cite[Proposition 2.1]{LWZ23}), $\int_X-v_j(\beta+dd^c\rho+dd^cu)^n$ is bounded by $||v_j||_{L^1(X,\omega)}$, so it is uniformly bounded (cf. \cite[Lemma 2.3]{DP10}).  Thus $E_{\chi}(v_j+\rho)$ is uniformly bounded.
			Let $\Phi_j:=(\sup_{k\geq j}v_k+\rho)^*$. By Proposition \ref {prop: preserve sup}, $$\langle(\beta+dd^c\Phi_j)^n\rangle\geq \inf_{k\geq j}\langle(\beta+dd^c\rho+dd^cv_k)^n\rangle\geq \min\{f,j\}\nu.$$
			
			Since $\Phi_j\searrow \rho+v$ and $\Phi_j\geq v_j+\rho$, by Proposition \ref{prop: fund inequ} and Remark \ref{rem: chi ene ineq}, we have $$\int_X(-\chi\circ (\Phi_j-\rho))\langle(\beta+dd^c\Phi_j)^n\rangle\leq 2^n\int_X(-\chi\circ v_j)\langle(\beta+dd^c\rho+dd^cv_j)^n\rangle.$$
			
			Thus, $\rho+v$ belongs to $\mathcal{E}_{\chi}(X,\beta)$ by Proposition \ref{prop:char echi}. Theorem \ref {thm: mono converge} implies $$\langle(\beta+dd^c\rho+dd^cv)^n\rangle=\lim_{j\rightarrow \infty}\langle(\beta+dd^c\Phi_j)^n\rangle\geq \lim_{j\rightarrow \infty}\min\{f,j\}\nu=\mu.$$
			
			Combining with $\int_X\langle(\beta+dd^c\rho+dd^cv)^n\rangle=\int_X\beta^n=\mu(X)$, we conclude that $$\langle(\beta+dd^c\rho+dd^cv)^n\rangle=\mu.$$
			
			The proof of Theorem \ref{thm: main 2} is complete.

		\end{proof}

	\end{document}